\documentclass[11pt,a4paper,leqno]{amsart}
\usepackage{amsmath,amsfonts}
\usepackage{amssymb}
\usepackage{upref}
\usepackage[mathcal]{eucal}

\date{\today}

\begin{document}

\title{Gaussian vectors with Markov property}
\author{Maciej P. Wojtkowski}
\address{University of Opole\\
 Opole, Poland}
\email{mwojtkowski@uni.opole.pl}
\date{\today}

            \theoremstyle{plain}
\newtheorem{lemma}{Lemma}
\newtheorem{proposition}[lemma]{Proposition}
\newtheorem{theorem}[lemma]{Theorem}
\newtheorem{corollary}[lemma]{Corollary}
\newtheorem{fact}[lemma]{Fact}
   \theoremstyle{definition}
\newtheorem{definition}{Definition}
    \theoremstyle{example}
\newtheorem{example}{Example}
    \theoremstyle{remark}
\newtheorem{remark}{Remark}

\newcommand{\macierz}[4]{\scriptsize{\left[\begin{array}{cc} #1 & #2 \\ #3 & #4 \end{array}\right]}}

\begin{abstract}
  We demonstrate the parallel between the properties of Gaussian vectors
  and the Euclidean geometry. In particular we study the Markov property
  and give various equivalent Euclidean and probabilistic characterizations.
  We also give a simple Euclidean proof of the conditional maximality of
  the differential entropy for the Markov Gaussian vector
  (related to the Burg's Theorem).    
\end{abstract}

\maketitle

\section{Introduction}\label{int}

It is a math folklore that the study of 
Gaussian vectors is in some respects part of Euclidean geometry.
In this note we demonstrate this parallel by exploring
the Markov property in a Gaussian vector.

We say that a random vector $X= (X_1,X_2,\dots,X_n)$ is {\it Markov}
if for every $i=3,\dots,n$, the conditional distributions  of 
$X_i | (X_{i-1},\dots,X_1)$ and $X_i |  X_{i-1}$ coincide
with probability $1$.

The Markov property seems to involve a distinguished ``arrow of time''
in a vector. It is well known that actually it can be reformulated
as the following perfectly symmetric property.

\begin{definition}\rm
A random vector is {\it Markov} if for every $i=2,\dots, n-1$
the random vectors $Y_i = (X_1,X_2,\dots, X_{i-1})|X_i$ and
$Z_i = (X_{i+1},X_{i+2},\dots, X_{n})|X_i$ are independent with probability
$1$.
\end{definition}

Let $X$ be a random vector in $\mathbb R^n$ with Gaussian distribution
with zero mean and the nondegenerate covariance matrix $C =\{c_{ij}\} $, with
the inverse matrix
$C^{-1} =A=\{a_{ij}\}$.
Let $R= \{\rho_{ij}\}$ be the matrix of correlation coefficients,
where $\rho_{ij}$ is the correlation coefficient of the random
variables $X_i$ and $X_j$. In particular $\rho_{ii}=1$ and
$R=D^{-1}CD^{-1}$ where $D$ is the diagonal matrix with standard deviations
of the random variables $X_i$ on the diagonal.


We will establish the following theorem.

\begin{theorem}
For a Gaussian vector $X$ the following properties are equivalent

(i) $X$ is Markov.

(ii) The conditional expected values
$E(X_i\ | \ X_{i-1},\dots,X_1)$ and $E(X_i\ | \ X_{i-1})$ coincide
with probability $1$, for every $i=3,\dots,n$.

(iii) For every $1\leq k < l\leq n$ the correlation coefficient  
\[
\rho_{kl} =\rho_{k,k+1}\rho_{k+1,k+2}\dots \rho_{l-2,l-1}\rho_{l-1,l}.
\]

(iv) The matrix $R^{-1}=DAD=\{g_{ij}\}$ is given by 
$g_{ij} = 0$ if $|i-j|\geq 2$,
$g_{ii} = \alpha_i$ for $i=1,\dots, n$,
$g_{k,k+1}=\beta_k$ for $k=1,\dots, n-1$, where  
\[
\begin{aligned}
&\alpha_1=\frac{1}{1-\rho^2_{1,2}},\ 
  \alpha_n=\frac{1}{1-\rho^2_{n-1,n}}, \
  \alpha_i = \frac{1-\rho^2_{i-1,i}\rho^2_{i,i+1}}
{\left(1-\rho^2_{i-1,i}\right)
  \left(1-\rho^2_{i,i+1}\right)},\\
&\text{for}
  \ \ 2\leq  i \leq n-1, \ \   \text{and } 
  \ 
\beta_k = \frac{-\rho_{k,k+1}}
     {1-\rho^2_{k,k+1}}  \ \text{for}
   \  1 \leq k \leq n-1.
    \end{aligned}
\]

(v)
The matrix $A=\{a_{ij}\}$ is tridiagonal, i.e., $a_{ij} =0$ if  $|i-j|\geq 2$.
\end{theorem}

Essential part of Theorem 1 was first published by Barrett and Feinsliver,
\cite{B-F}, although it was probably known to probabilists earlier. 
In Section 2 we  give a probabilistic proof of this Theorem.

The structure of the inverse of a tridiagonal symmetric matrix was described
earlier in the theorem of Gantmacher and Krein, \cite{B},
and these ideas found extensive  following in
linear algebra, \cite{V-B-G-M}.

In Section 3 we provide the dictionary allowing
the translation from probabilistic  language to the
geometric  language. 

In Section 4 we  reformulate Theorem 1 in the language of euclidean
geometry. We also give a purely geometric proof taken mostly from \cite{W},
where it was first observed.

In Section 5 we also discuss the following  theorem about maximization
of differential entropy, which is fairly easy to prove in the geometric
language.
\begin{theorem}
  Among all random vectors $X$ with distribution densities
  $f_X \in L^1(\mathbb R^n)$, with fixed variances
  $cov(X_1,X_1), cov(X_2,X_2), \dots, cov(X_{n},X_n)$, and 
  fixed correlation coefficients
  $\rho(X_1,X_2), \rho(X_2,X_3), \dots, \rho(X_{n-1},X_n)$,
  different from $\pm 1$, 
  the differential entropy
  \[
  dent(f_X) = -\int_{\mathbb R^n}f_X(x)\ln f_X(x)dx
  \]
  is maximal for the Gaussian density which is Markov.

  The maximum is strict up to the translations in $\mathbb R^n$,
  i.e., up to  the expectations of the random vector.
\end{theorem}
Clearly the Markov Gaussian density is uniquely defined
by the formulas in Theorem 1 (iv) for the elements of the tridiagonal matrix
$A = C^{-1}$.

A version of Theorem 2  appears in the proof of
the Burg's Theorem given by Choi and Cover, \cite{C-C}.

\section{Probability}
We will prove Theorem 1.

$(i) \implies (ii)$
is the direct consequence of Markov property.

$(ii) \implies (iii)$

We use linear regression of $X_l$ on $X_{1},X_2,\dots.X_{l-1}$, i.e.,
we represent

$X_l=\eta_1X_{1}+\eta_{2}X_{2}+\dots +\eta_{l-1}X_{l-1}+W_l$ where
$Cov(X_{i},W_l)=0$ for $i\leq l-1$.

Since the random vector $(X_{1},X_{2},\dots, X_{l-1},W_l)$
has Gaussian distribution
we conclude that $(X_{1},X_{2},\dots, X_{l-1})$ and $W_{l}$
are independent. It follows that 
$E(W_l|X_1,X_2,\dots,X_{l-1}) = E(W_l) =0$.
  It gives us 
  \[
\begin{aligned}
  &E(X_l|X_1,X_2,\dots,X_{l-1}) =
  \\ & \eta_1X_{1}+\eta_{2}X_{2}+\dots +\eta_{l-1}X_{l-1}+
  E(W_{l}|X_1,X_2,\dots,X_{l-1})
  \\ =
  &\eta_1X_{1}+\eta_{2}X_{2}+\dots +\eta_{l-1}X_{l-1}
\end{aligned}
\]

By the assumption (ii)
$E(X_l|X_1,X_2,\dots,X_{l-1}) = E(X_l|X_{l-1})$, and hence
$\eta_1 =\eta_{2} =\dots = \eta_{l-2} =0$ and
$E(X_l|X_1,X_2,\dots,X_{l-1}) = \eta_{l-1}X_{l-1}$.

We calculate now

$E(X_{l-1}X_l) = E\left(E(X_{l-1}X_l|X_{l-1})\right) =
E\left(X_{l-1}E(X_l|X_{l-1})\right)=
\eta_{l-1}\sigma_{l-1}^2$ which leads to
$\sigma_{l-1}\eta_{l-1}= \sigma_{l}\rho_{l-1,l}$.

Finally for $k\leq l-1$
\[
\begin{aligned}
  &E(X_kX_l) = E\left(E(X_kX_l|X_{1},X_{2},\dots,X_{l-1})\right) =
  E\left(X_kE(X_l|X_{1},X_{2},\dots,X_{l-1})\right) =
\\ &=
\eta_{l-1}E\left(X_{k}X_{l-1})\right).
\end{aligned}
\]

We obtained $\rho_{k,l}\sigma_k\sigma_l =
\eta_{l-1}\rho_{k,l-1}\sigma_{k}\sigma_{l-1}$
and thus $\rho_{k,l} = \rho_{k,l-1}\rho_{l-1,l}$,
We conclude the proof by the induction on the distance of
$k$ and $l$.

$(iii) \implies (iv)$

We need to check that the matrix product $H=\{h_{ij}\} $
of the symmetric matrix described in (iv)
and the symmetric matrix $R$ 
 is equal to identity.
It takes a lot of checking, however the algebra is
rather straightforward. We will do it for some choice of
indices only. For $2\leq i < j \leq n$ we have 
\[
\begin{aligned}  
h_{ij}= -\frac{\rho_{i-1,i}\rho_{i-1,j}}{1-\rho_{i-1,i}^2}
+\frac{(1-\rho_{i-1,i}^2\rho_{i,i+1}^2)\rho_{ij}}
{(1-\rho_{i-1,i}^2)(1-\rho_{i,i+1}^2)}
-\frac{\rho_{i,i+1}\rho_{i+1,j}}
{1-\rho_{i,i+1}^2}
\end{aligned}
\]
Using (iii) we identify the common factor $\rho_{ij}$.
Extracting it we are left with the following sum, which is clearly
equal to $0$
\[
\begin{aligned}  
\frac{h_{ij}}{\rho_{ij}}= -\frac{\rho_{i-1,i}^2}{1-\rho_{i-1,i}^2}
+\frac{1-\rho_{i-1,i}^2\rho_{i,i+1}^2}
{(1-\rho_{i-1,i}^2)(1-\rho_{i,i+1}^2)}
-\frac{1}{1-\rho_{i,i+1}^2}=0.
\end{aligned}
\]
Further for $2\leq i  \leq n-1$
\[
\begin{aligned}  
h_{ii}= -\frac{\rho^2_{i-1,i}}{1-\rho_{i-1,i}^2}
+\frac{1-\rho_{i-1,i}^2\rho_{i,i+1}^2}
{(1-\rho_{i-1,i}^2)
  (1-\rho_{i,i+1}^2)}
-\frac{\rho_{i,i+1}^2}
{1-\rho_{i,i+1}^2}=1.
\end{aligned}
\]
We leave the remaining cases to the reader.
They will be also dealt with in the euclidean language in Section 2.

$(iv) \implies (v)$ is obvious.

$(v) \implies (i)$

Let us describe the conditional distribution of the
random vector

$(X_1,\dots, X_{k-1},X_{k+1},\dots, X_{n})|X_k$.
It is a Gaussian distribution with the density equal to
\[
f_k =f_k(x_1,\dots, x_{k-1},x_{k+1},\dots, x_{n}|x_k)=
const \exp(-2^{-1}\sum_{i,j=1}^na_{ij}x_ix_j)
\]
considered  as a function of $n-1$ variables
$(x_1,\dots, x_{k-1},x_{k+1},,\dots, x_{n})$, with 
$x_k$ treated  as a  parameter. Note that for this function
to be a density the constant in front must be an appropriate
function of the parameter $x_k$.
Since the matrix $A$ is tridiagonal we have
\[
f_k =
const \exp(-2^{-1}\sum_{i,j=1}^{k}a_{ij}x_ix_j)
\exp(-2^{-1}\sum_{i,j=k}^na_{ij}x_ix_j)
\exp(2^{-1}a_{kk}x_k^2)
\]
We represented the density as a product of a function
of variables $(x_1,\dots, x_{k-1})$
and a function of variables $(x_{k+1},,\dots, x_{n})$
(with $x_k$ treated as a parameter). By the well known
criterion we conclude that the random vectors
$(X_1,\dots, X_{k-1})|X_k$
and $(X_{k+1},\dots, X_{n})|X_k$
are independent.

Theorem 1 is proven.

\section{The dictionary}

  In this section we spell out the dictionary which allows the translation
  from probability to Euclidean geometry, and back, for Gaussian
  random vectors.

  In the background we have a probability space $\Omega$ with
  the $\sigma$-algebra of events $\mathcal F$, and the probability $P$.
  We consider the Hilbert space $L_0^2=L^2_0(\Omega,\mathcal F,P)$,
  of random variables with zero expected values and finite second moments.
  The inner product in $L_0^2$ is furnished by the covariance
  $Cov(\cdot, \cdot)$
  of random variables.

  A random vector $X= (X_1,X_2,\dots,X_n)$ with zero expected values
  is an array of random variables 
  in $L^2_0$.

  Assuming linear independence of the random variables we consider
  the $n$-dimensional subspace
  $Z = span\{ X_1,X_2,\dots,X_n\}
  \subset L^2_0$, which is isometric with
  the $n$-dimensional Euclidean space.

On the Euclidean side 
  we consider a sequence of linearly independent vectors $(e_1,e_2,\dots,e_n)$
  in a Euclidean space $\mathbb E^n$ with the scalar product
  $\langle \cdot, \cdot \rangle$.

Under the assumption   that 
the random vector $X= (X_1,X_2,\dots,X_n)$
is Gaussian the Euclidean data 
$C= \{c_{i,j}\} = \{Cov(X_i,X_j)\}$ gives complete information
about the joint distribution of the random vector (in $\mathbb R^n$), namely
\[
g(x) = \left((2\pi)^n\det C \right)^{-\frac{1}{2}}\exp
\left(2^{-1}\sum_{i,j=1}^{n}a_{ij}x_ix_j\right)
\]
is the density function of the distribution of a  Gaussian vector
$G$ with values in
$\mathbb R^n$, zero mean
and the covariance matrix equal to $C =A^{-1}$.

Any two random variables $ Y_0, Y_1$    from the space 
$Z = span\{ X_1,X_2,\dots,X_n\}$
have   Gaussian joint distribution. It follows that 
the orthogonality of the random variables $ Y_0, Y_1$  
equals zero correlation, which equals independence.

Further the conditional expected value
$E(Y_0|Y_1)$ is a random variable from $Z$  
  equal to the  orthogonal projection of $Y_0$
onto the line (the $1$-dimensional subspace) spanned by $Y_1$.

For a sequence $(Y_0,Y_1,Y_2,\dots,Y_k)$ of random variables from the space
$Z$ the conditional expected value
$ E(Y_0|Y_1,Y_2,\dots,Y_k)$ is the linear combination
of random variables $Y_1,Y_2,\dots,Y_k$,
  equal to the   orthogonal projection of $Y_0$
onto the subspace spanned by $Y_1,Y_2,\dots, Y_k$.

Finally we describe the conditional distribution of
$Y_0|Y_1,Y_2,\dots,Y_k$ in Euclidean terms.
It is a family of distributions which coincide after
centering, and the common centered (i.e., with $0$ expected value)
distribution is
that of $W = Y_0- E(Y_0|Y_1,\dots,Y_k)$. Indeed
$W$ is orthogonal to  $span(Y_1,\dots,Y_k)$,
and hence independent of $Y_1,Y_2,\dots, Y_k$.

Let us note that $W$ is the orthogonal projection
of $Y_0$ onto 
the orthogonal complement of $span(Y_1,\dots,Y_k)$.

\section{Euclidean geometry}
Theorem 1 has a Euclidean geometry counterpart.
Let $(e_1,e_2,\dots, e_n)$ be an ordered linear basis of unit vectors,
in general not  orthogonal, in $\mathbb R^n$ with the scalar
product $\langle \cdot,\cdot\rangle$. Let
$(f_1,f_2,\dots, f_n)$ be the dual basis, i.e.,
$\langle f_i,e_j\rangle =\delta_{ij}$, the Kronecker delta.
Further for $k=1,\dots,n$, let $S_k$ be the linear subspace spanned by the first
$k$ vectors $\{e_1,e_2,\dots,e_k\}$, and $T_k$ the linear subspace
spanned by the last $k$ vectors  $\{e_{n-k+1},\dots,e_n\}$.
The sequences of subspaces 
\[
\{0\}=S_0 \subset S_1 \subset S_2 \subset \dots
\subset S_n =\mathbb R^n, 
\{0\}=T_0 \subset T_1 \subset T_2 \subset \dots
\subset T_n =\mathbb R^n, 
\]
form two {\it flags} which have the additional property that
for every $k=1,\dots,n$ the subspaces $S_k$ and $T_{n-k}$ are
transversal. Any two flags $\mathcal S=(S_0,S_1,\dots,S_n)$
and $\mathcal T =(T_0,T_1,\dots,T_n)$ will be called {\it compatible flags}
if they have this property.

It is transparent that for any two compatible flags $\mathcal S$
and $\mathcal T$ we can
choose an ordered basis $(\pm e_1,\pm e_2,\dots, \pm e_n)$
of unit vectors so that
the flags are obtained from the basis as above. The first vector 
$e_1$ is a generator of $S_1$, and the only freedom we have is to choose
between $e_1$ and $-e_1$. Since we assume that the two flags are
compatible we get that for every $k=1,\dots ,n$ the subspace
$S_k\cap T_{n-k+1}$ has dimension $1$. The vector $e_k$ is a generator
of  the one dimensional subspace $S_k\cap T_{n-k+1}$,
and we can choose between $e_k$ and $-e_k$.
We will say that any ordered  basis of unit vectors
$(\pm e_1,\pm e_2,\dots, \pm e_n)$ is {\it adapted} to
the two compatible flags $\mathcal S, \mathcal T$.

We need to consider a family of  orthogonal projections 
associated with the subspaces in two compatible flags.
For $k= 1,\dots n-1$  we denote by $P_k$ and $Q_k$ 
the orthogonal projections onto the subspaces $S_k$ and $T_k$,
respectively. By $P_k' = I-P_k,\ Q_k' = I-Q_k$ we denote the
orthogonal projections onto the complementary subspaces.
It is clear that $P_k'$ is the orthogonal projection onto the
span of $\{f_{k+1},\dots,f_n\}$, and $Q_{n-k}'$ the orthogonal projection onto the
span of $\{f_{1},\dots,f_{k}\}$.

Further we introduce the orthogonal projections $M_k$ onto
the one dimensional subspaces $S_k\cap T_{n-k+1}$, spanned by $e_k$.
Again $M_k'=I-M_k$ is the orthogonal projection
onto the complementary subspace, which is the
span of $\{f_1,\dots, f_{k-1},f_{k+1},\dots,f_n\}$.

Now we have two compatible flags $\mathcal S, \mathcal T$,
with an ordered  basis of unit vectors  $( e_1, e_2,\dots, e_n)$
adapted to them, and the dual ordered basis $(f_1,f_2,\dots, f_n)$.

\begin{proposition}
The following are equivalent

(a) For $1\leq k \leq n-1$ and for any vector $v\in T_{n-k}$, 
  $P_kv = M_kv$.

(b) For $1\leq k \leq n-1$ the vector $P_ke_{k+1}$ is parallel to $e_k$.

(c) For $2\leq k \leq n-1$ the subspaces $M_k'S_{k-1}$
      and $M_k'T_{n-k}$ are orthogonal.

      (d) For $1\leq k \leq n-1$ and for any vector $v\in S_{k}$,
      $Q_{n-k}v = M_{k+1}v$.

(e) For $1\leq k \leq n-1$ the vector $Q_{n-k}e_{k}$ is parallel to $e_{k+1}$.
\end{proposition}

We will prove the equivalence of (a), (b) and (c).
Since the role of the two flags is perfectly
symmetric in (c), it follows immediately that (c), (d) and (e)
are also equivalent.

To clarify the meaning of this  Proposition, and before we prove
it in the geometric language, let us translate it into
the language of Probability, using our dictionary.

\begin{proposition}
  For a nondegenerate Gaussian random vector $(X_1,X_2,\dots,X_n)$
  the following are equivalent

(a) For $1\leq k \leq n-1$ and for any $Y\in span\{ X_{k+1},\dots,X_n\}$
\[
E(Y |  X_1,X_2,\dots, X_{k}) = E(Y |  X_{k}).
\]

(b) For $2\leq k \leq n-1$
\[
E(X_{k+1} |  X_1,X_2,\dots, X_{k}) = aX_k
\]
for some scalar $a$.

(c) For $2\leq k \leq n-1$ and any $Y_1 \in span\{ X_{1},\dots,X_{k-1}\}$
and $Y_2 \in span\{ X_{k+1},\dots,X_{n}\}$
the conditional random variables $Y_1|X_k$ and $Y_2|X_k$
are independent with probability $1$.

(d) For $1\leq k \leq n-1$ and for any $Y\in span\{ X_{1},\dots,X_{k}\}$
\[
E(Y |  X_{k+1},X_{k+2},\dots, X_{n}) = E(Y |  X_{k+1}).
\]

(e) For $2\leq k \leq n-1$
\[
E(X_{k} |  X_{k+1},X_{k+2},\ dots, X_{n}) = aX_{k+1}
\]
for some scalar $a$.
\end{proposition}

\begin{proof}(of Proposition 3)
  It is obvious that (a) implies (b).

 To prove (c) we observe that for any $j > k$, using (b), we get
$P_ke_j = P_kP_{k+1}\dots P_{j-1}e_j= ae_k$, for some scalar $a$.
Further 
$P_ke_j =ae_k = M_k e_j$. Indeed we can prove by induction
on the distance between the indices $k$ and $j$ that
\[
  e_{j} = a e_k + a_{k+1}f_{k+1}+a_{k+2}f_{k+2}+\dots + a_{n}f_{n},
  \]
  for some scalars $a_{k+1},\dots, a_n$.

  Further $M_ke_j+M_k'e_j= e_j= P_ke_j+P_k'e_j$. It follows that  
  \[
  M_k'e_j= a_{k+1}f_{k+1}+a_{k+2}f_{k+2}+\dots + a_{n}f_{n}=P_k'e_j,
  \]
  and hence $M_k'e_j$ is orthogonal to $S_k\supset M_k'S_k
  \supset M_k'S_{k-1}$. To summarize, we have established for any
  $e_j, j\geq k+1$,
  that $M_k'e_j$ is orthogonal to $M_k'S_{k-1}$. It follows that 
  the whole subspace $M_k'T_{n-k}$ is orthogonal to $M_k'S_{k-1}$.

To obtain (a) from (c) we observe that for $v\in T_{n-k}$, 
$P_kv = P_k(M_kv +M_k'v) = M_kv + P_k(M_k'v)$. Since $M_k'v$ is orthogonal to
$e_k$ and to $M_k'S_{k-1}$, then $M_k'v$ is orthogonal to $S_k$. Hence  
$P_k(M_k'v) =0$.
\end{proof}

We will say that an ordered basis $( e_1, e_2,\dots, e_n)$, or the
compatible pair of flags $\mathcal S,\mathcal T$, is {\it Markov}
if it satisfies any of the equivalent properties in Proposition 3.

\begin{theorem}
The following are equivalent
  
  (1) The ordered basis of unit vectors   $( e_1, e_2,\dots, e_n)$
  is Markov.
  
(2) For every $1\leq k < l\leq n$   
\[
\langle e_k,e_l\rangle = \langle e_{k},e_{k+1}\rangle\langle e_{k+1},e_{k+2}\rangle
\dots \langle e_{l-2},e_{l-1}\rangle\langle e_{l-1},e_{l}\rangle.
\]

(3) The dual basis $(f_1,f_2,\dots, f_n)$ is equal to
\[
f_i = \beta_{i-1}e_{i-1}+\alpha_{i}e_i+\beta_ie_{i+1}, i =1,\dots, n,
\]
where we use the convention that $e_0=e_{n+1}=0$ and the coefficients 
\[
\alpha_i = \frac{1-\langle e_{i-1},e_{i}\rangle^2
  \langle e_{i},e_{i+1}\rangle^2}
{\left(1-\langle e_{i-1},e_{i}\rangle^2\right)
  \left(1-\langle e_{i},e_{i+1}\rangle^2\right)},\
\beta_i = \frac{-\langle e_{i},e_{i+1}\rangle}
{1-\langle e_{i},e_{i+1}\rangle^2},
\]

(4) $\langle f_{i},f_{j}\rangle=0$ if $|i-j| \geq 2$.
\end{theorem}
\begin{proof}
  $(1)\implies (2)$
  We will use the property (e) in Proposition 2. It follows that
  for any $1\leq k < l \leq n$ we have
  $e_k = Q_{n-k}e_k+Q_{n-k}'e_k$ with  $Q_{n-k}e_k =
  \langle e_{k},e_{k+1}\rangle e_{k+1}$. Since 
   $Q_{n-k}'e_k$ is orthogonal to $e_l$ we conclude that
  $\langle e_{k},e_{l}\rangle=
  \langle e_{k},e_{k+1}\rangle
  \langle e_{k+1},e_{l}\rangle$. Now we obtain (2) by induction on the
  distance of $k$ and $l$.

  $(2)\implies (3)$ 
  By the uniqueness of the dual basis
  it is sufficient to check that with the given formulas for $f_i$ 
  we have $\langle f_{i},e_{j}\rangle = \delta_{ij}$. The algebra is
  the same as in the proof of Theorem 1.
  We start with the calculation of $  \langle f_{i},e_{i}\rangle$
for $1\leq i\leq n$. We get
\[
\begin{aligned} 
\frac{-\langle e_{i-1},e_{i}\rangle^2}
{1-\langle e_{i-1},e_{i}\rangle^2}+\frac{1-\langle e_{i-1},e_{i}\rangle^2
  \langle e_{i},e_{i+1}\rangle^2}
{(1-\langle e_{i-1},e_{i}\rangle^2)
  (1-\langle e_{i},e_{i+1}\rangle^2)}
+\frac{-\langle e_{i},e_{i+1}\rangle^2}
{1-\langle e_{i},e_{i+1}\rangle^2}=1.
\end{aligned}
\]
We proceed with the calculation of   $\langle f_{i},e_{j}\rangle$ for
$1\leq i< j\leq n$. It is equal to  
\[
\begin{aligned}
-\frac{\langle e_{i-1},e_{i}\rangle\langle e_{i-1},e_{j}\rangle}
{1-\langle e_{i-1},e_{i}\rangle^2}+\frac{\left(1-\langle e_{i-1},e_{i}\rangle^2
  \langle e_{i},e_{i+1}\rangle^2\right) \langle e_{i},e_{j}\rangle}
{\left(1-\langle e_{i-1},e_{i}\rangle^2\right)
  \left(1-\langle e_{i},e_{i+1}\rangle^2\right)}
-\frac{\langle e_{i},e_{i+1}\rangle\langle e_{i+1},e_{j}\rangle}
{1-\langle e_{i},e_{i+1}\rangle^2}.
\end{aligned}
\]
Using (2) we identify the common factor $\langle e_{i},e_{j}\rangle$.
Extracting it we are left with the following sum, which is clearly
equal to $0$
\[
\begin{aligned}
-\frac{\langle e_{i-1},e_{i}\rangle^2}
{1-\langle e_{i-1},e_{i}\rangle^2}+\frac{1-\langle e_{i-1},e_{i}\rangle^2
  \langle e_{i},e_{i+1}\rangle^2}
{\left(1-\langle e_{i-1},e_{i}\rangle^2\right)
  \left(1-\langle e_{i},e_{i+1}\rangle^2\right)}
-\frac{1}
{1-\langle e_{i},e_{i+1}\rangle^2}=0.
\end{aligned}
\]

$(3)\implies (4)$ 
  In general for dual bases 
  $( e_1, e_2,\dots, e_n)$ and $(f_1,f_2,\dots, f_n)$
  we have
  \[
  f_{i}= \sum_{j=1}^n\langle f_{i},f_{j}\rangle e_j,\ \ \ \ \ \ \ \ \ \ \ (*)
  \]
  and the Gram matrix $\{\langle f_{i},f_{j}\rangle\}$ is the inverse
  of $\{\langle e_{i},e_{j}\rangle\}$.
 Hence (4) follows immediately from (3).

 $(4)\implies (1)$

 We will prove (b) of Proposition 3 by induction on the index $k$. 
 We have that $f_n$ is orthogonal to $S_{n-1}$ and using $(*)$ we get 
 $\langle f_{n},f_{n}\rangle e_n =f_n - \langle f_{n},f_{n-1}\rangle e_{n-1}$.
 Applying the projection  $P_{n-1}$ to both sides we get
 $\langle f_{n},f_{n}\rangle P_{n-1}e_n =- \langle f_{n},f_{n-1}\rangle e_{n-1}$.
 So we have (b) for $k=n-1$.

 Recall that $P_{k}$  is the orthogonal projection onto the span
 of vectors $\{e_{1},\dots,e_{k}\}$, and $f_{k+1}$ is orthogonal
 to all these vectors.
 Given that $P_{k}e_{k+1}=ae_{k}$ for some coefficient $a$,
 we apply the projection $P_{k}$ to both sides of the equation
\[ 
f_{k} =\langle f_{k},f_{k-1}\rangle e_{k-1}+\langle f_{k},f_{k}\rangle e_{k}+
\langle f_{k},f_{k+1}\rangle e_{k+1}.
\]
We obtain 
\[ 
P_{k}f_{k} = \langle f_{k},f_{k-1}\rangle e_{k-1}+
\left(\langle f_{k},f_{k}\rangle+\langle f_{k},f_{k+1}\rangle a\right)
e_{k}.
\]
The coefficient with $e_{k}$ cannot be $0$ since otherwise
we would have

$P_{k}f_{k} =
\langle f_{k},f_{k-1}\rangle e_{k-1}$, and further
\[
0 = \langle f_{k}, e_{k-1}\rangle=
\langle f_{k}, P_{k}e_{k-1}\rangle=\langle P_{k}f_{k}, e_{k-1}\rangle= 
\langle f_{k},f_{k-1}\rangle,
\]
so that $P_{k}f_{k} = 0$.  We get the contradictory conclusion 
that $f_{k}$ is orthogonal to $e_{k}$.

We can thus write $P_{k}f_{k} =
\langle f_{k},f_{k-1}\rangle e_{k-1}+be_{k}$, for some coefficient $b\neq 0$.
Applying $P_{k-1}$ to both sides of the last equation we get 
$0=\langle f_{k},f_{k-1}\rangle e_{k-1}+bP_{k-1}e_{k}$, which completes
the induction.

\end{proof}

\section{Differential entropy}

For a random vector $X$ with a density function
$f\in L_1(\mathbb R^n)$ the {\it differential entropy}
is equal to
\[
dent(f) = -\int_{\mathbb R^n} f(x)\ln f(x)dx,
\]
if $f\ln f \in L_1(\mathbb R^n)$ (otherwise it is not defined).

The differential entropy is not changed if we add a constant vector to $X$,
so without loss of generality we can assume that the random vector has
zero expected values.

It is  a well known fact that the Gaussian distribution
maximizes the differential entropy among all densities with the same
covariance matrix $C$. For the convenience of the reader we provide
a complete proof essentially taken from the book by Cover and Thomas,
 \cite{C-T}.

Let $g(x) = (\sqrt {(2\pi)^n\det C })^{-1}
\exp\left(-\frac{1}{2}\langle Ax,x\rangle\right)$ be the density function
of a Gaussian vector $G$ with values in
$\mathbb R^n$, zero mean
and the covariance matrix equal to $C =A^{-1}$.
\begin{lemma}
  If the covariance matrices of a random vector $X$ with density $f$
  and the Gaussian vector $G$ with density $g$ coincide then
  \[
  -\int_{\mathbb R^n} f(x)\ln g(x)dx =   dent(g) =
  \frac{1}{2}\left(n\ln(2\pi e) + \ln\det C\right).
  \]
\end{lemma}
\begin{proof}
  Integrating the quadratic form $\langle Ax,x\rangle = \sum_{i,j}a_{i,j}x_ix_j$
  against the density $f$ we get
  \[
 \sum_{i,j}a_{ij}\int_{\mathbb R^n} f(x)x_ix_jdx = \sum_{i,j}a_{ij}\int_{\mathbb R^n} g(x)x_ix_jdx =  \sum_{i,j}a_{ij}c_{ij} = n.
 \]
 We conclude that 
\[
\begin{aligned}
&  -\int_{\mathbb R^n} f(x)\ln g(x)dx =\frac{1}{2}(n\ln(2\pi) + \ln\det C) +
  \frac{1}{2} \sum_{i,j}a_{ij}\int_{\mathbb R^n} f(x)x_ix_jdx\\ 
  &=   \frac{1}{2}(n\ln(2\pi e) + \ln\det C).
\end{aligned}
\]
\end{proof}

We are ready to prove that the Gaussian distribution
maximizes the differential entropy among all densities with the same
covariance matrix $C$.

\begin{proposition}
  If a random vector $X$ with density $f$ has the same covariance matrix
  $C$ as the
  Gaussian vector $G$ with density $g$ then
  \[
  dent(f) \leq dent(g)
  \]
  and the equality holds only if $f =g$.
\end{proposition}
\begin{proof}
  Given a density function $f$
  let us introduce an auxiliary function $h(x) =\frac{g(x)}{f(x)}$, 
well defined in the support subset $S$ of the density $f$, i.e.,
$S =\{x\in \mathbb R^n| f(x) >0\}$. We extend the definition of $h(x)$
to all of $\mathbb R^n$ by assigning it a fixed value, say $1$, outside of $S$.
We have $E(h(X)) = \int_{S} f(x) \frac{g(x)}{f(x)}dx =
\int_{S}  g(x)dx \leq 1$.
We apply the Jensen inequality to the convex function $-\ln(y)$ and
the random variable $Y= h(X)$. Using Lemma 5 we get
\[
0\leq -\ln E(Y) \leq E(-\ln(Y)) =  -\int_{S} f(x)\ln \frac{g(x)}{f(x)}dx =
dent(g) -dent(f).
\]
To end the proof we recall that  for a strictly
convex function the Jensen inequality becomes equality
only if the random variable is constant.
\end{proof}

Finally we discuss random vectors $X$ with zero mean, and  with prescribed 
variances $c_{ii}=Cov(X_i,X_{i}), i = 1,2,\dots,n$
  and covariances $c_{i,i+1} = Cov(X_i,X_{i+1}), i = 1,2,\dots,n-1$.
We will prove Theorem 2, which says that 
among all such random vectors $X$
the differential entropy is maximal for the unique Markov Gaussian
  vector.
\begin{proof}
  By Proposition 7 it is sufficient to compare differential
  entropies only for Gaussian vectors. By Lemma 6 we need to maximize
  the determinant of the covariance matrix. By rescaling we can restrict
  our attention to random vectors with all variances equal to 1. 
 
  Translating to the euclidean language we have a sequence of unit vectors
  $(e_1,e_2,\dots ,e_n)$ with prescribed scalar products
  $\langle e_i,e_{i+1}\rangle$ and we want to establish what is the maximum
  volume of the parallelopiped spanned by these vectors.
  We achieve it by the induction on the dimension $n$.  For $n=2$
  the area is actually constant and it is equal to
  $V_2=\sqrt{1-\langle e_1,e_{2}\rangle^2}$. The Markov property for 
$n=2$ is void. 
  
  Let us assume that
  the Proposition holds
  for some $n\geq 2$, and that the above maximal volume is
  equal to $V_n=\sqrt{\prod_{i=1}^{n-1}
    \left(1-\langle e_i,e_{i+1}\rangle^2\right)}$.
  Let $e_{n+1} = u+v$, where $u$ is the orthogonal projection of $e_{n+1}$
  onto the subspace $S_n$, i.e., the span of $\{e_1,e_2,\dots, e_n\}$.
  It is clear that the maximal volume for $n+1$ unit vectors
  $\{e_1,e_2,\dots, e_n,e_{n+1}\}$
  is equal to $V_n$ times the maximal possible length of the vector $v$.
  Since $\langle e_n,e_{n+1}\rangle$ is fixed, the maximum length of the
  vector $v$ is equal to $\sqrt{1-\langle e_n,e_{n+1}\rangle^2}$
  and it is assumed only for vectors $e_{n+1}$ for which the orthogonal
  projection $u$ is parallel to $e_n$. Hence by induction we obtain
  the Markov property for the sequence of unit vectors with the maximum
  volume.
  \end{proof}

\end{document}